\documentclass[12pt, colorlinks=true,allcolors=blue]{article}

\usepackage{todonotes}

\usepackage{amsmath}

\usepackage{amsfonts}
\usepackage{amssymb}
\usepackage{graphicx}
\usepackage{doi}
\usepackage{amsthm}

\usepackage{xcolor}
\usepackage{todonotes}
\usepackage{hyperref}

\usepackage{xfrac}

\usepackage{mathtools}

\newtheorem{theorem}{Theorem}
\newtheorem{lemma}[theorem]{Lemma}
\newtheorem{definition}[theorem]{Definition}
\newtheorem{corollary}[theorem]{Corollary}
\newtheorem{observation}[theorem]{Observation}

\newtheorem{proposition}[theorem]{Proposition}
\newtheorem{remark}[theorem]{Remark}

\theoremstyle{remark}

\newcommand{\conv}{\text{conv}}

\title{Approximate cycle double cover\thanks{Supported by the
Czech Science Foundation project Flows, cycles, surfaces and polynomials
(grant number 25-16627S) and by the project GAUK182623 of the Charles University Grant Agency.}}
\author{Babak Ghanbari, Robert Šámal \\
Computer Science Institute of Charles University \\
\{babak, samal\}@iuuk.mff.cuni.cz
}
\date{\today}

\begin{document}
\maketitle

\begin{abstract}
    The Cycle double cover (CDC) conjecture states that for every bridgeless graph $G$,
there exists a family $\mathcal{F}$ of cycles such that each edge of the graph
is contained in exactly two members of $\mathcal{F}$. Given an embedding of a
graph~$G$, an edge $e$ is called a \emph{singular edge} if it is visited twice by the
boundary of one face. The CDC conjecture is equivalent to bridgeless cubic graphs
having an embedding with no singular edge. In this work, we introduce nontrivial
upper bounds on the minimum number of singular edges in an embedding of a cubic
graph. Moreover, we present efficient algorithms to find embeddings satisfying these 
bounds. 
\end{abstract}


\section{Introduction}
The Cycle double cover conjecture (CDC for short) states that for every bridgeless graph $G$,
there exists a family $\mathcal{F}$ of cycles such that each edge of the graph
is contained in exactly two members of $\mathcal{F}$. 
It was made independently by Szekeres
\cite{szekeres_1973} and Seymour \cite{seymour1979sums} in the 70s and is now
widely considered to be among the most important open problems in graph theory.
A reason for this is its close relationship with topological graph theory,
integer flow theory, graph coloring, and the structure of snarks. More detail
about the conjecture and historical notes can be found in \cite{zhang_2012}. 

Here, we mention just some of the partial results. 
It is known (and easy) that the CDC conjecture is true for Hamiltonian cubic graphs, more generally for $3$-edge-colorable cubic graphs. 
Tarsi~\cite{Tarsi} and Goddyn~\cite{Goddyn}
proved that graphs with a Hamilton path admit a 6-CDC, a CDC with six cycles (or, more precisely, 
cycles that can be combined into six Eulerian graphs). 
Goddyn~\cite{Goddyn} and Häggkvist, McGuinness~\cite{HMcG} construct a CDC, among else, 
in graphs that have a spanning subgraph that is a Kotzig graph.

For a $2$-factor $S$ of $G$, the oddness of $S$, denoted by $odd(S)$, is the
number of odd components of $S$. For a cubic graph $G$, the oddness of~$G$, denoted by $odd(G)$, is the minimum of $odd(S)$ for all $2$-factors~$S$ of
$G$. The determination of the oddness of a cubic graph is a hard problem since
already the determination of $3$-edge-colorability of a cubic graph is an
NP-complete problem \cite{Holyer} and a cubic graph is $3$-edge-colorable if and
only if it has oddness~$0$. Graphs with low oddness have drawn significant
interest in the study of graph 
theory, especially in relation to the CDC conjecture. Huck and Kochol
\cite{HUCK1995119} showed that every bridgeless cubic graph with oddness at most $2$
has a $5$-CDC. This result was further improved by Huck \cite{HUCK2001125}, and
by Häggkvist, McGuinness~\cite{HMcG}, independently, for oddness $4$ graphs. 

It is known~{\cite[Section 1.2]{zhang_2012}} that it is enough to prove the CDC conjecture for bridgeless cubic
graphs. It is also known~{\cite[Section 1.4]{zhang_2012}} that for a bridgeless cubic graph, each cycle double cover
consists of facial boundaries of some embedding. However, not every embedding gives a cycle double 
cover -- a face boundary may visit one edge twice. We will call such edges 
\emph{singular} (following \cite{Mohar}). Thus, the CDC conjecture is equivalent to every bridgeless cubic
graph having an embedding with no singular edges. We approximate
this statement by looking for an embedding with not too many singular edges.

This has a precedent: 
Bender and Richmond \cite{Bender} estimate the number of singular edges using Euler characteristics. They showed that an embedding of a $3$-edge-connected graph on a surface other than the projective plane and sphere has at most $3-3\chi$ singular edges where $\chi$ is the Euler characteristic of the surface. For the $3$-edge-connected graphs embedded on the projective plane or the sphere, the number of singular edges is at most $1$ or~$0$, respectively. They also claim (without proof) a $k(1-\chi)/(k-2)$ bound for the case of $k$-edge-connected graphs.

One method to compute the Euler characteristic of a surface is by determining its genus. For a simple graph $G$, let $g$ denote its genus, defined as the smallest integer $h$ for which $G$ can be embedded in the orientable surface~$\mathbb{S}_h$ of genus $h$. Similarly, let $\tilde{g}$ be the non-orientable genus of $G$, which is the smallest integer $c$ such that $G$ can be embedded in the non-orientable surface~$\mathbb{N}_c$ with crosscap number $c$.  

It is well known that for an orientable embedding, the Euler characteristic satisfies $\chi = 2 - 2g$, while for a non-orientable embedding, it follows that~$\chi~=~2~-~\tilde{g}.$ Applying these formulas along with Bender and Richmond's bound on the number of singular edges, we can reformulate their bound as follows.

\begin{corollary}[\cite{Bender}]\label{num-singular-Bender}
    The number of singular edges in an orientable embedding of genus $g$ is at most $6g - 3$. The number of singular edges in a non-orientable embedding of genus $\tilde{g}$ is at most $3\tilde{g}- 3$.
\end{corollary}

Determining the minimum genus of a graph is generally an NP-complete problem \cite{THOMASSEN-general}, and this remains true even for cubic graphs \cite{THOMASSEN-cubic}. However, in special cases, such as complete graphs, the exact minimum genus is known~\cite{ringel2012map}. We show that our bounds on the number of singular edges are stronger than those obtained by Bender and Richmond by providing an example of a cubic graph with a large genus and then comparing the bounds on the number of singular edges from each method. To this end, we start with a complete graph $K_n$ and extend it to a cubic graph $G_n$ by replacing each vertex of $K_n$ with a circuit $C_{n-1}$
(this construction is not unique). The resulting graph $G_n$ has $n(n-1)$ vertices. Each vertex of the circuit is incident with two edges within the circuit and one edge from $K_n$ (see \hyperref[G_n]{Figure~1}). Therefore, $G_n$ is cubic, and since $K_n$ is a minor of $G_n$, the following result follows immediately as a corollary of Ringel's bound on the genus of complete graphs~\cite{ringel2012map}.

\begin{figure}\label{G_n}
    \begin{center}
        \includegraphics[scale=.18]{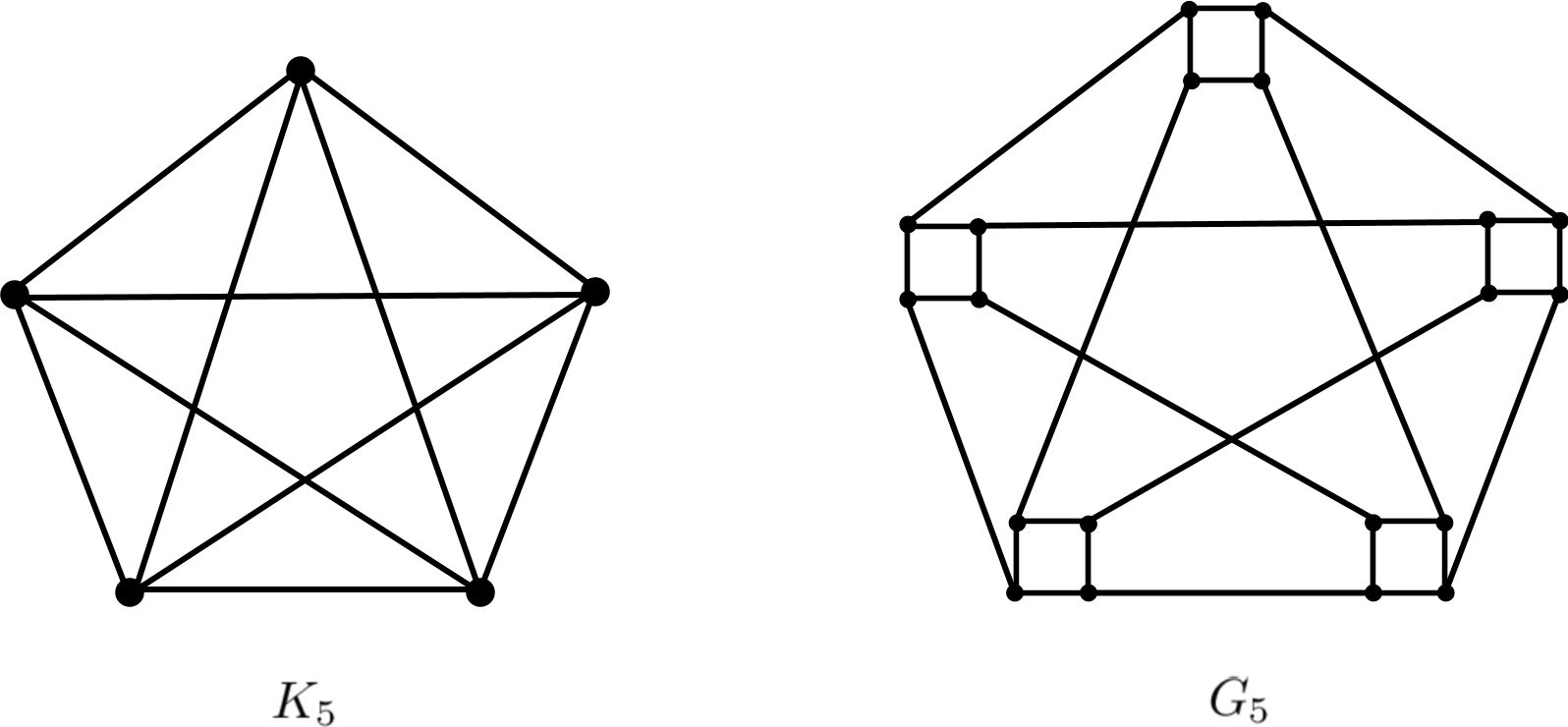}
    \end{center}
    \caption{An example of graphs $K_5$ and $G_5$}
\end{figure}

\begin{corollary}\label{cor-genus-G_n}
    If $n\geq 3$, then
        $$g(G_n) \geq g(K_n) = \left\lceil\dfrac{(n - 3)(n - 4)}{12}\right\rceil.$$ 
    If $n \geq 5$ and $n \neq 7$, then
        $$\tilde{g}(G_n) \geq \tilde{g}(K_n) = \left\lceil\dfrac{(n - 3)(n - 4)}{6}\right\rceil.$$
\end{corollary}

The best bound on the number of singular edges that 
\hyperref[num-singular-Bender]{Corollary~\ref{num-singular-Bender}} can give is therefore 
$ \frac{n^2}{2}(1 - o(1)) $, both in the orientable and non-orientable cases. However, since $G_n$ has $n(n-1)$ vertices, by our \hyperref[obs19]{Theorem~\ref{obs19}}, we get a better bound of $\frac{n^2}{10}$ which is one advantage of our method compared to  Bender and Richmond's method.


We prove nontrivial upper bounds on the minimum number of singular edges in an
embedding of a cubic graph in \hyperref[obs17]{Theorem~\ref{obs17}},
\hyperref[obs19]{Theorem~\ref{obs19}}, and \hyperref[thm6]{Theorem~\ref{thm6}}.
Our main technique is using a \emph{partial CDC}, a collection of cycles that 
cover every edge once or twice (by two different cycles) such that no 
``angle'' (pair of adjacent edges) is contained in two distinct cycles. 
In \hyperref[lem7]{Lemma~\ref{lem7}},  we show that a partial CDC can always 
be extended to an embedding of the graph that gives an embedding with many regular edges
(all edges covered by the partial CDC will be regular edges in the embedding). 
It would be interesting to know what the limit of this technique is -- how many 
singular edges there must be in an extension of a given partial CDC. A known example 
showing that we cannot generate a CDC by extending an arbitrary partial CDC 
is the Petersen graph with a 2-factor consisting of two 5-circuits (see \hyperref[CDC-petersen]{Observation~\ref{CDC-petersen}}).

The complexity of testing whether a given graph has a cycle double cover is unknown. 
However, it is presumed that it can be done in a polynomial time, as we only need to 
check whether the graph has any bridges. Finding a concrete CDC in a given graph may be nontrivial, even if the CDC conjecture holds true.
We start the exploration 
of this question by giving most of our results in an algorithmic version: we provide
polynomial time algorithms for finding (in different contexts) an embedding 
of a given cubic graph where the number of singular edges is bounded by a linear function
of the number of vertices. 

Given a cubic graph $G$, and a $3$-edge-coloring for $G$, it is trivial to find
a CDC because a 3-edge-colorable cubic graph is doubly covered by the three
bi-colored $2$-factors (see \cite[Theorem 1.3.2]{zhang_2012}). However, it is
hard to find a $3$-edge-coloring for the graph, even if it is cubic. 
It is well known that the decision problem for $3$-edge-coloring is NP-complete
\cite{Holyer}. We now explain that also the problem of finding a
$3$-edge-coloring of a $3$-edge-colorable graph is hard.

Suppose we have an algorithm $A$ that for any cubic $3$-edge-colorable graph~$G$
on $n$ vertices finds a $3$-edge-coloring in at most~$f(n)$ steps, where $f(n)$ 
is a polynomial. We run $A$ and count steps. 
If $A$ finds a coloring, we verify it (for graphs that are not $3$-edge-colorable, the algorithm can do anything). If we fail to find a coloring after $f(n)$~steps of $A$, we declare the graph not to be 3-edge-colorable. 
It is easy to see that this algorithm would decide $3$-edge-colorability and 
that the overhead for counting of steps (and verification of the coloring found)
will keep the algorithm in polynomial time. 

In the case of a Hamiltonian cubic graph, finding a CDC is straightforward if we are given a Hamiltonian cycle: 
first, a $3$-edge-coloring can be obtained by using two
colors to color the Hamiltonian cycle and assigning the third color to the
remaining edges. As previously mentioned, a $3$-edge-coloring yields a CDC easily. 
However, it is NP-complete to decide if a graph has a Hamiltonian cycle \cite{Garey}. 
As in the previous paragraph, it follows it is not possible to find a Hamiltonian cycle 
in a cubic graph in a polynomial time, even if we know it exists. 
As a result, finding a CDC remains difficult even when the graph is Hamiltonian.

By the above discussion, we wanted to illustrate that even in the context of the CDC, it is relevant to indicate how efficient our methods are. We state all our results including the time complexity of an algorithm that finds our object of interest: a collection of closed walks with a bounded number of singular edges. 


The structure of the paper is as follows: In the next section, we
provide the necessary definitions. In Section 3, we introduce the notion of a
partial CDC and state a key lemma, which is utilized in the subsequent section
to prove our main theorems. In Section 4, we present our main theorems and their
proofs. Finally, we conclude with a summary in the final section.

\section{Preliminaries} 
We follow the notation of the book Graphs on Surfaces \cite{Mohar}. Here, we list the definitions that are most important to us.

A graph in which all vertices have degree $3$ is called \emph{cubic}. A \emph{circuit} is 
a connected $2$-regular graph. A \emph{cycle} is a graph in which the degree of each vertex 
is even. A \emph{bridge} in a graph $G$ is an edge whose removal increases the number of 
components of $G$. Equivalently, a bridge is an edge that is not contained in any circuit 
of $G$. A graph is \emph{bridgeless} if it contains no bridge.
A connected graph is \emph{$k$-edge-connected} if it remains connected whenever fewer than $k$ edges are removed. 
A graph is \emph{cyclically $k$-edge-connected}, if at least $k$ edges must be removed to disconnect 
it into two components such that each component contains a cycle.

A graph $G$ is \emph{embedded} in a surface $S$ if the vertices of $G$ are distinct elements of $S$ and every edge of $G$ is a simple arc connecting in $S$ the two vertices which it joins in $G$, such that its interior is disjoint from other edges and vertices. An \emph{embedding} of a graph $G$ in $S$ is an isomorphism of $G$ with a graph $G'$ embedded in $S$. If there is an embedding of $G$ into $S$, we say that $G$ can be embedded into $S$. 

Let $G$ be a graph that is cellularly embedded in a surface $S$, that is, the interior of every face is homeomorphic to an open disk. Let $\pi = \{\pi_v | v \in V(G)\}$ where $\pi_v$ is the cyclic permutation of the edges incident with the vertex $v$ such that $\pi_v(e)$ is the successor of $e$ in the clockwise ordering around $v$. The cyclic permutation $\pi_v$ is called the \emph{local rotation} at $v$, and the set $\pi$ is the \emph{rotation system} of the given embedding of $G$ in $S$. 

Let $G$ be a connected multigraph. A \emph{combinatorial embedding} of $G$ is a
pair $(\pi, \lambda)$ where $\pi = \{\pi_v | v \in V(G)\}$ is a rotation system,
and $\lambda$ is a signature mapping which assigns to each edge $e \in E(G)$ a
sign $\lambda(e) \in \{-1, 1\}$. If $e$ is an edge incident with $v \in V(G)$,
then the cyclic sequence $e, \pi_v(e), \pi_v^2(e), \dots$ is called the
$\pi$-\emph{clockwise ordering} around $v$ (or the \emph{local rotation} at
$v$). Given an embedding $(\pi, \lambda)$ of $G$ we say that $G$ is 
$(\pi, \lambda)$-\emph{embedded}.

It is known that the combinatorial embedding uniquely determines a cellular embedding to some surface, up to homeomorphism~{\cite[Theorems 3.2.4 and 3.3.1]{Mohar}}.

A \emph{closed walk} is a sequence $(v_0, e_0, v_1,$ $e_1, \dots, e_{n-1}, v_n)$ where $v_0 = v_n$ and for every $i$, $e_i = \{v_i,v_{i+1}\}$.

\begin{definition}\label{def4}
Let $(\pi, \lambda)$ be an embedding of a graph $G$. A \emph{$(\pi, \lambda)$-facial walk} (or \emph{facial walk}) $F$ is a closed walk $(v_0, e_0, v_1, e_1, \dots,$ $ e_{n-1}, v_n)$ where  
$v_0 = v_n$, $e_0 = e_n$, and
for every $i \in \{1, \dots, n\}$,
\[e_i = 
  \begin{cases} 
   \pi_{v_i}(e_{i-1}) & \text{if}~~~   \mu_{i} = 1\\
   \pi^{-1}_{v_i}(e_{i-1}) & \text{if}~~~ \mu_{i} = -1
  \end{cases}
\] 
where $\mu_i = \varepsilon\lambda(e_0)\lambda(e_1)\dots \lambda(e_{i})$ and $\varepsilon \in \{-1, +1\}$ determines the side of the edge $e_0$ that is being used by $F$. 
\end{definition}

\begin{observation}[\cite{Mohar}]\label{obs:face_exists}
    Let $(\pi, \lambda)$ be an embedding of a graph~$G$. For every edge $e$ of~$G$ and every $\varepsilon \in \{-1, +1\}$ there is a facial walk described in 
    \hyperref[def4]{Definition~\ref{def4}}. 
    When we consider cyclic shifts and reversals of a facial-walk to be the same walk, then 
    every edge appears once in two different walks or twice in the same facial walk. 
\end{observation}

\section{Partial CDC}

Let $G = (V,E)$ be a cubic and bridgeless graph, $(\pi,\lambda)$ be an embedding of~$G$, and $\mathcal{F} = \{{F_1}, \dots, {F_t}\}$ be the set of all facial
walks in this embedding (we take just one facial walk among cyclic shifts and reversals). 
As mentioned in \hyperref[obs:face_exists]{Observation~\ref{obs:face_exists}}, each edge of the graph is covered twice by one or once by two members of $\mathcal{F}$. We say that an edge $e$ is
\emph{singular} if the former case occurs, that is, if there exists a facial walk $ {F_i} \in \mathcal{F}$ that uses $e$ twice. Otherwise, it is called a
\emph{regular} edge. We say that $\mathcal{F}$ is a \emph{facial double cover} (or FDC) of $G$. 

Edges $e$ and $f$ incident with vertex $v$ are $\pi$-\emph{consecutive} if $\pi_v(e) = f$ or $\pi_v(f) = e$. Every 
such pair $\{e, f\}$ of edges forms a $\pi$-\emph{angle}. We 
say that a collection of closed walks $C_1, \dots, C_n$ forms a 
\emph{partial circuit double cover} (or \emph{partial CDC}) 
if 
\begin{enumerate}
    \item[\textbf{C1}.]\label{C1} Each edge is covered at most once by one of $C_i$'s or exactly twice by two different $C_i$ and $C_j$.
    \item[\textbf{C2}.]\label{C2} For every vertex $v$ and edges $e$ and $f$ where 
        $e \cap f = \{v\}$, there exists at most one closed walk $C_i$ such that $\{e, f\} \subseteq E(C_i)$.
\end{enumerate}
In other words, condition \hyperref[C2]{C2} means that no angle is used twice. The following lemma explains the term partial CDC.

\begin{lemma}\label{lem7}
Let $G$ be a bridgeless cubic graph, and $C_1, C_2, \dots, C_t$ be a collection of closed walks in $G$. If 
$C_1, C_2, \dots, C_t$ form a partial CDC, 
then there is an embedding $(\pi, \lambda)$ of $G$ where $C_1, \dots, C_t$ are some of the facial walks of $(\pi, \lambda)$.
Moreover, such an embedding can be found by a linear time algorithm. 
\end{lemma}

    \begin{proof}
        Let $C_1, C_2, \dots, C_t$ be a partial CDC. To extend $C_1, C_2, \dots, C_t$ to an embedding, we 
        proceed as follows. For every $v \in V(G)$ we create a graph~$D_v$ with $V(D_v) = \delta(v)$,
        i.e., vertices of $D_v$ are the three edges incident to the vertex $v$. For every $i$ and every 
        $e,f \ni v$ if $C_i$ contains $e,v,f$ then we add the edge $\{e,f\}$ to $D_v$. Since $G$ is cubic, condition (\hyperref[C1]{C1}) implies that $\Delta(D_v)\leq 2$, and (\hyperref[C2]{C2}) implies that it has no double edges. 
        Now, for every $v\in V(G)$ we extend $D_v$ to a circuit and fix an orientation 
        $\overrightarrow{D_v}$ of it. If $(e,f)$ is an arc of $\overrightarrow{D_v}$, we define 
        $\pi_v(e) = f$. To create an embedding $(\pi, \lambda)$, we will define $\lambda$ as 
        follows. For every $e\in E(G)$ if $e$ is not used by any $C_i$, define $\lambda(e) = 1$. 
        Otherwise, let $C_i$ contain $f,u,e,v,g$ where $f,e,g \in E(G)$ and $u,v \in V(G)$. 
        \[\lambda(e) = 
            \begin{cases} 
                1 & \text{if}~~~   \pi_u(f) = e ~~ \& ~~ \pi_v(e) = g \\

                -1 & \text{if}~~~ \pi_u(f) = e ~~ \& ~~ \pi^{-1}_v(e) = g \\

                -1 & \text{if}~~~  \pi^{-1}_u(f) = e ~~ \& ~~ \pi_v(e) = g \\

                1 & \text{if}~~~ \pi^{-1}_u(f) = e ~~ \& ~~ \pi^{-1}_v(e) = g
            \end{cases}
        \]
        It remains to check that $\lambda$ is well defined. We need to show that if
        $e \in E(C_i) \cap E(C_j)$, then $C_i$ and $C_j$ define the same $\lambda(e)$. This is easy to check. Note that comparing $C_i$ with $C_j$ gives two changes. For example, if in $C_i$ we have $\pi_u(f) = e$ and $\pi^{-1}_v(e) = g$ then $\lambda(e) = -1$. Now, let $\delta(u) = \{e, f, f'\}$, $\delta(v) = \{e, g, g'\}$. Since $C_1, C_2, \dots, C_t$ is a partial CDC, by conditions (\hyperref[C1]{C1}) and (\hyperref[C2]{C2}), the walk $C_j$ must contain $f',u,e,v,g'$ with $\pi^{-1}_u(f') = e$ and $\pi_v(e) = g'$ (or alternatively, $C_j$ may be visited in the other direction as $g', v, e, u, f'$ i.e. $\pi_u(e) = f'$ and $\pi^{-1}_v(g') = e$) which gives $\lambda(e) = -1$ (See \hyperref[MainLemma]{Figure~2}). Similarly, the remaining three cases of the defined $\lambda(e)$ by $C_i$ give the same $\lambda(e)$ in $C_j$.

        The algorithm proceeds as the proof above: First, we read the description of all of the cycles $C_1$, \dots, $C_t$ and 
        construct the graphs $D_v$ along the way. Then, we extend each $D_v$ to a $3$-circuit and orient it (constant time for each~$v$), 
        which defines $\pi$. 
        Finally, we read the descriptions of the cycles again and define~$\lambda$ by the formula above. 
    \end{proof}
    \begin{figure}\label{MainLemma}
        \begin{center}
            \includegraphics[scale=.4]{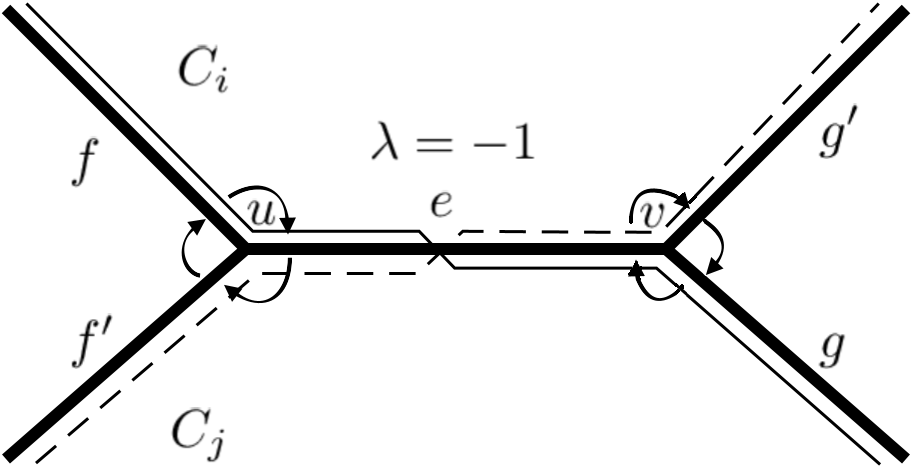}
        \end{center}
        \caption{$C_i$ represented by a narrow line and $C_j$
        represented by a dashed line. $f,u,e,v,g \in C_i$ and $f',u,e,v,g'\in C_j$}
    \end{figure}

The following known observation shows that not every partial CDC can be extended to a CDC.
We will prove it by using the fact that no angle is used twice in a CDC.
\begin{observation}\label{CDC-petersen}
    A 2-factor in Petersen graph consisting of two $5$-circuits forms a partial CDC that can not be extended to a CDC.
\end{observation}
    \begin{proof}
        We have two $5$-circuits $C_1$ and $C_2$ and a perfect matching $M$ between them. 
        No angle in a CDC is used twice. Since $C_1$ and $C_2$ are disjoint, they satisfy 
        conditions (\hyperref[C1]{C1}) and (\hyperref[C2]{C2}). Therefore, $C_1$ and 
        $C_2$ forms a partial CDC. Now, we show that in any extension of $C_1$ and $C_2$ 
        to an FDC $\mathcal{F}$, there exists at least one singular edge. If there is no singular 
        edge, then every facial walk must be a circuit. Note that $C_1$ and $C_2$ cover $10$ edges together. 
        Since every edge is covered twice in a CDC, the remaining facial walks must cover $20$ edges. Let 
        $F$ be a facial walk in an extension $\mathcal{F}$ of $C_1$ and $C_2$ to an FDC. Every angle that is not used by $C_1$ and 
        $C_2$, contains an edge from $M$. So, the edges of $F$ alternate between $C_1$, 
        $M$, $C_2$, $M$, $C_1$, $\dots$ Therefore, the length of $F$ is divisible by $4$. 
        Since there is no $4$-circuit and $12$-circuit in Petersen, the length of $F$ must be 
        $8$. But it is not possible to cover exactly $20$ edges with $8$-circuits, a contradiction.
    \end{proof}

\begin{lemma}
For every FDC of an embedding of a graph $G$, each angle is part of some facial walk.
\end{lemma}
    \begin{proof}
        Let $\{e, f\}$ be a $\pi$-angle in an FDC where $e\cap f = \{v\}$. By definition of $\pi$-angle, we have $\pi_{v}(e) = f$ or $\pi_{v}(f) = e$. Assume that $\pi_{v}(e) = f$ and let $e = (u,v)$ and $f = (v, w)$. Then $F = (u, e, v, f, w, \dots, u)$ with $\epsilon = \lambda(e)$ is a facial walk in this FDC by
        ~\hyperref[obs:face_exists]{Observation \ref{obs:face_exists}}.
    \end{proof}

We use the notation of the book Combinatorial Optimization \cite{Schrijver} to denote perfect matching polytopes. 
In an undirected graph $G = (V, E)$ and $U\subseteq V$, we use $\delta(U)$ to denote the set of edges leaving $U$.  
For any set $Y$, we identify the function $x: Y \longrightarrow \mathbb{R}$ with the vectors $x\in \mathbb{R}^Y$. Equivalently, we denote $x(v)$ by $x_v$. For $U \subseteq Y$, we denote
$x(U):=\sum\limits_{v\in U}^{}x_v$. For a matching~$M$, we use $\chi_M$ for the incidence vector of $M$ in $\mathbb{R}^E$. Then $\mathcal{P}_{PM}(G)$, the perfect matching polytope of a graph~$G$,  is defined as follows
$$ \mathcal{P}_{PM}(G) = \conv\{ \chi_M : M~\text{is a perfect matching in}~G \}.       $$

\begin{theorem}[Edmonds, \cite{Edmonds1965MaximumMA}] \label{Edmonds}

\[
\mathcal{P}_{PM}(G) =
   \begin{Bmatrix*}[l]
                        &x\geq 0,\\
    x\in \mathbb{R}^E : &\forall v\in V:~ x(\delta(v)) = 1,\\
                        &\forall U\subseteq V: ~|U| ~ \text{is odd}~\Rightarrow 
                        x(\delta(U)) \geq 1
   \end{Bmatrix*}
\]

\end{theorem}

In the next section, we use Edmonds' theorem to get upper bounds on the minimum number of singular edges in the embedding of bridgeless cubic graphs.

\section{Bounds on the minimum number of singular edges using perfect matchings}\label{sec3}


By \hyperref[lem7]{Lemma \ref{lem7}}, every partial CDC $C_1, C_2, \dots,$ $C_t$ can be extended to an FDC. Since in a partial CDC, each edge is either covered once by one facial walk or twice by two facial walks, the edges that are already covered can not be singular. Therefore, potential singular edges of the FDC are those that are not covered by $C_1, C_2, \dots,$ $C_t$. We will use this observation repeatedly.

\begin{theorem}\label{obs17}
Let $G$ be a bridgeless cubic graph. There exists an embedding of $G$ with at most $\frac{n}{2}$ singular edges.
Moreover, such an embedding can be found in time $O(n \log^2 n)$. 
\end{theorem}

\begin{proof}
    Let us start with any perfect matching $M$. (This can be found in time $O(n \log^2 n)$, see \cite{Diks}.) 
    Then, $G-M$ is a cycle, that is, a collection of circuits. By \hyperref[lem7]{Lemma \ref{lem7}}, this collection can be extended to a double cover for $G$ in linear time. Therefore, the resulting embedding has at most $|M| = \frac{n}{2}$ singular edges. 
\end{proof}

The following theorem of Kaiser, Kr\'{a}l, and Norine is proved using \hyperref[Edmonds]{Theorem~
\ref{Edmonds}}. This result is formulated only as existence of matchings $M_1$ and $M_2$. However, 
it is proved as formulated below.
We will use it to improve upon \hyperref[obs17]{Theorem \ref{obs17}}.

\begin{theorem}[\cite{Kaiser}]\label{thm18}
Let $G$ be a bridgeless cubic graph, and let $M_1$ be a perfect matching in $G$ that contains no $3$-edge-cuts in $G$. 
Then, there exists a perfect matching $M_2$ such that $|M_1 \cup M_2| \geq \frac{3}{5}|E(G)|$.
\end{theorem}

We reformulate and extend this theorem as below.

\begin{corollary}\label{cl18}
Let $G$ be a bridgeless cubic graph, and let $M_1$ be a perfect matching in $G$ that contains no $3$-edge-cuts in $G$. Then, there exists a perfect matching $M_2$ in~$G$ with $|M_1 \cap M_2| \leq \frac{n}{10}$. Moreover, such $M_2$ can be found in time $O(n^\frac{3}{2}\log n)$. 
\end{corollary}
    \begin{proof}
        By \hyperref[thm18]{Theorem \ref{thm18}}, there exists a perfect matching $M_2$ such that 
        $$|M_1 \cup M_2| \geq \frac{3}{5}|E(G)| = \frac{3}{5}\cdot\frac{3n}{2} = \frac{9n}{10}.$$ 
        On the other hand, 
        $$|M_1 \cup M_2| = |M_1| + |M_2| - |M_1 \cap M_2| = n - |M_1 \cap M_2|.$$
        Therefore, 
        $$|M_1 \cap M_2| \leq n - \frac{9n}{10} = \frac{n}{10}.$$
        This proves the existence of $M_2$. To find it efficiently, we use the algorithm of \cite{Duan2018} to solve 
        minimum weight perfect matching (the weights being~$\chi_{M_1}$) in time $O(n^\frac{3}{2}\log n)$.
    \end{proof}

\begin{theorem}\label{obs19}
Let $G$ be a bridgeless cubic graph. There exists an embedding of $G$ with at most $\frac{n}{10}$ singular edges. Moreover, such an embedding can be found in time $O(n^3)$.
\end{theorem}

\begin{proof}
    Let $M_1$ and $M_2$ be two perfect matchings that are provided by \hyperref[thm18]{Theorem~\ref{thm18}}. Let $C_1 = G - M_1$ and $C_2 = (M_1 \cup M_2) - (M_1 \cap M_2)$.
    Next, we will show that $C_1$, $C_2$ is a partial CDC. Each edge of $M_2 \backslash M_1$ is used by both $C_1$ and $C_2$. Each edge 
    of $M_1 \backslash M_2$ is used by $C_2$, and each edge of $E(G)\backslash (M_1 \cup M_2)$ is used by $C_1$. The remaining edges are not used by any of $C_1$ or $C_2$. Therefore, each edge is covered at most twice by $C_1$ and $C_2$. To prove that no angle is used twice by $C_1$ and $C_2$, 
    let $e, f, g$ be three edges incident with a vertex $v$. Since $v$ is covered by an edge in both $M_1$ and $M_2$, we have two following cases 

    \begin{enumerate}
        \item\label{case1} One of $e$, $f$, and $g$, let say $e$, is covered by both of $M_1$ and $M_2$. In this case, angle $\{f, g\}$ is used by $C_1$, and other angles are not used (see \hyperref[fig]{Fig. \ref{fig}}).
        \item\label{case2}Two of $e$, $f$, and $g$ are covered, one by $M_1$ and the other by $M_2$. Let $e \in M_1$ and $f \in M_2$. In this case, angle $\{e, g\}$ is not used by any of $C_1$ and $C_2$. Angle $\{e, f\}$ is used by $C_2$, and angle $\{f, g\}$ is used by $C_1$. Since $g \notin E(C_2)$, angle $\{f, g\}$ is not used by $C_2$, and since $e \notin E(C_1)$, angle $\{e, f\}$ is not used by $C_1$ 
        (see \hyperref[fig]{Fig. \ref{fig}}).
        \begin{figure}\label{fig}
            \begin{center}
                \includegraphics[scale=.2]{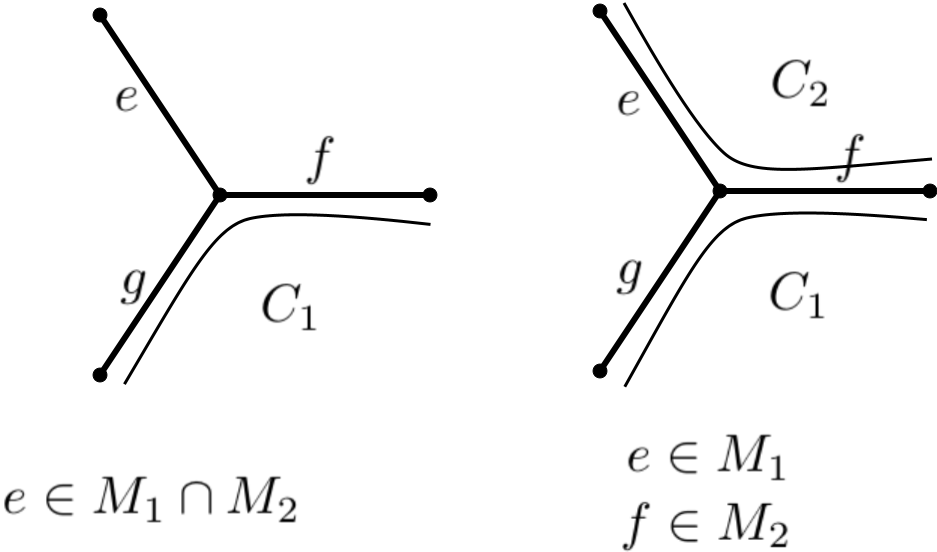}
            \end{center}
            \caption{Case \ref{case1} (left picture) and case \ref{case2} (right picture) in the proof of \hyperref[obs19]{Theorem~\ref{obs19}}.}
        \end{figure}
    \end{enumerate}
    Therefore, no angle is used twice, and the collection of circuits forming $C_1$ and $C_2$ is a partial CDC.  By \hyperref[lem7]{Lemma~\ref{lem7}} and 
    \hyperref[cl18]{Corollary~\ref{cl18}}, the resulting embedding from extending $C_1$ and $C_2$ has at 
    most $|E(G) \backslash (C_1 \cup C_2)| = |M_1 \cap M_2| \leq \frac{n}{10}$ singular edges.

    To prove the time complexity, it is enough to find perfect matchings $M_1$ and $M_2$ in time $O(n^3)$.
    Then, by \hyperref[lem7]{Lemma~\ref{lem7}}, we can extend 
    $C_1 = G - M_1$ and $C_2 = (M_1 \cup M_2) - (M_1 \cap M_2)$
    to an embedding in linear time. In \cite{Boyd}, Boyed et al. provided an algorithm that finds a 
    minimum-weight perfect matching in $G$, which intersects every $3$-edge cut in $G$ in exactly 
    one edge in time $O(n^3)$. Then, we use the Edmonds' blossom algorithm to find $M_2$ in time
    $O(n^3)$.
\end{proof}

\begin{remark}
    A forthcoming paper \cite{ghanbari2025time} of the authors improves the time complexity  of finding $M_1$ in the above theorem to $O(n \log^4 n)$ for $3$-edge-connected cubic graphs. Then we can use the faster algorithm from \hyperref[cl18]{Corollary~\ref{cl18}} to achieve a time complexity of $O(n^\frac{3}{2}\log n)$. 
\end{remark}

Given a graph $G = (V, E)$, a \emph{postman set} is a set $J \subseteq E$ such that the multi-set $J \cup E$ is Eulerian, that is, every vertex of the graph is incident with even number of edges. The following known proposition is immediate.

\begin{proposition}
    $J$ is a postman set of $G = (V, E)$ if and only if the following condition holds for every vertex $v\in V$:
    
    $v$ is incident to an odd number of edges in $J$ if and only if $v$ has odd degree in $G$.
\end{proposition}

Therefore, in a cubic graph $G = (V, E)$, a subset $J\subseteq E$ is a 
postman set if for every $v\in V(G)$, $v$ is incident with $1$ or $3$ edges 
in $J$. The following lemma is known. We include its 
proof for the reader's convenience.

\begin{lemma}\label{postman-SP}
    For every spanning tree $T$ in $G$ there exists a postman set $J$ such that $J \subseteq E(T)$. Moreover, such a postman set can be found in linear time.
\end{lemma}

\begin{proof}
        Let $J = \emptyset$. For each vertex 
        $v \in V(T)$, label $v$ with 
        $1$ if the degree of $v$ in $G$ is odd, and label it $0$ otherwise. Then start with a leaf 
        $v \in V(T)$. Let $e = uv$ be the edge incident with $v$. Let $l_v$ be the label of $v$.
        If $l_v = 1$, add $e$ to the set $J$, relabel the vertex $u$, the other end of $e$, with the 
        new label $l'_u = l_u + 1~(\text{mod}~2)$ and continue with the tree $T' = T - v$.
        If $l_v = 0$, then simply remove $v$ from the tree and continue the same process with the 
        tree $T' = T - v$ inductively.
        This can be done in time $O(n)$ where $n = |V(G)|$.
        First we create a list $L$ of leaves of $T$, this can be done in time $O(n)$. Then, 
        in each step of the above process, we pick a vertex $v$ from the list, and create the tree $T'$. 
        We update $L$ to get a list of leaves of $T'$: we remove $v$ and possibly add $u$ to $L$ (if $\deg_{T'}u=1$). 
        The time complexity of this 
        part is also $O(n)$ as the number of steps is at most $|V(T)|$.  
        Note that in the above process, every time we handle $v$,  
        a leaf of $T$, the degree of $v$ in the new multi-graph $G' = (V(G), E(G) \cup J)$ becomes 
        even. So the final graph $G'$ will be an Eulerian graph and therefore $J$ is a postman set.
\end{proof}

\begin{lemma}\label{postman-CDC}
    Let $M$ be a perfect matching and $J$ be a postman set in a bridgeless cubic graph $G$. Then, $C_1 = G - J$ and $C_2 = M\triangle J$ form a partial CDC.
\end{lemma}
\begin{proof}
    It is easy to check that $G-J$ and $M\triangle J$ (the symmetric difference) are cycles. Next, we prove that these cycles form a partial CDC. Each edge of $M \backslash J$ is used by both $C_1$ and $C_2$. Each edge 
    of $J \backslash M$ is used by $C_2$, and each edge of $E(G)\backslash (M \cup J)$ is used by $C_1$. The remaining edges ($J \cap M$) are not used by any of $C_1$ or $C_2$. Therefore, each edge is covered at most twice by $C_1$ and $C_2$. 
    To prove that $C_1$ and $C_2$ form a partial CDC, it remains to show that no angle is used twice by $C_1$ and $C_2$. Let $e$, $f$, and $g$ be three edges incident with a vertex $v$. Every vertex $v \in V(G)$ is covered by one edge from $M$ and either one or three edges from $J$. Therefore, we have the following two cases. 

    \begin{enumerate}
        \item There is only one edge from $J$ that is incident with $v$. The proof of this case is the same as the proof of case $1$ and $2$ in the proof of \hyperref[obs19]{Theorem~\ref{obs19}}. Therefore, no angle is used twice.   
        \item Vertex $v$ is incident with three edges from $J$. Let $e$ be the edge that belongs to the perfect matching $M$. 
        In this case, $e \in M\cap J$ and none of $C_1$ nor $C_2$ uses $e$. Therefore, angles $\{e,f\}$ and $\{e,g\}$ are used neither by $C_1$ nor by $C_2$. Both $f$ and $g$ are not in $M$ and are used by $C_2$. Therefore, the angle $\{f,g\}$ is used once by $C_2$. This completes the proof as no angle is used twice by $C_1$ and $C_2$.
    \end{enumerate}
\end{proof}

\begin{figure}\label{fig3}
    \begin{center}
        \includegraphics[scale=.2]{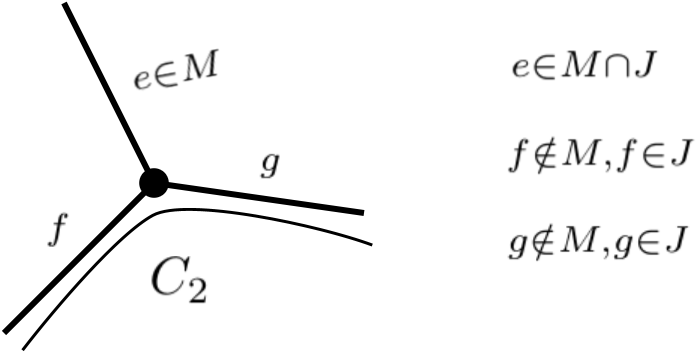}
    \end{center}
    \caption{Case $2$ in the proof of \hyperref[postman-CDC]{Lemma~\ref{postman-CDC}}.}
\end{figure}

\begin{theorem} \label{cyclically 2k-edge-connected}
Let $G$ be a cyclically $2k$-edge-connected cubic graph. There exists an embedding of $G$ with at most $\frac{n}{2k}$ singular edges. Moreover, such an embedding can be found in time
$\tilde{O}((kn)^\frac{3}{2})$.
\end{theorem}
    \begin{proof}
        Construct a multigraph $H$ from $G$ as follows: select any perfect matching $M$ of $G$ and remove the edges of $M$ from $G$. It can be done in time $O(n \log^2 n)$ \cite{Diks}. Now, contract each component of $G-M$ to a vertex. Bring back the edges of $M$ and name the new graph~$H$. Note that $E(H)$ is in natural $1-\text{to}-1$ correspondence with $M$.
        Clearly,~$H$ is $2k$-edge-connected as the graph $G$ is cyclically $2k$-edge-connected. By Nash-Williams$/$Tutte's Theorem~\cite{nash1961edge,tutte}, 
        there exist $k$ edge-disjoint spanning trees in~$G$ (This part can be done in time 
        $\tilde{O}((kn)^\frac{3}{2})$ \cite{Quanrud}, also see Page 879 of Schrijver's book.) Therefore, some 
        spanning tree in $H$ has at most 
        $$\frac{|E(H)|}{k} = \frac{|M|}{k} = \frac{n}{2k}$$
        edges. 
        Let $T_H$ be such spanning tree in $H$ (the smallest of these $k$) and let $T_G$ be any extension of 
        $T_H$ to a spanning tree of $G$. We use \hyperref[postman-SP]{Lemma~\ref{postman-SP}} to 
        get a postman set $J\subseteq E(T_G)$ in linear time. By \hyperref[postman-CDC]{Lemma~\ref{postman-CDC}}, 
        $C_1 = G-J$ and 
        $C_2 = M\triangle J$ form a partial CDC. 
        By \hyperref[lem7]{Lemma~\ref{lem7}}, we can extend $C_1$ and $C_2$
        to an embedding in linear time.
        Then, the number of singular edges is at 
        most  $|M \cap J|$ as potential singular edges are those that are not covered 
        by any of $C_1$ and $C_2$. We know that $J \subseteq E(T_G)$ 
        and $M \cap J \subseteq E(T_H)$ (using the $1-\text{to}-1$ correspondence of $E(H)$ with $M$). By our choice of $T_H$ we get
        $$|M\cap J| \leq |E(T_H)| \leq \frac{n}{2k}.$$
        Altogether, we can find such an embedding in time $\tilde{O}((kn)^\frac{3}{2})$. 
    \end{proof}

We improve our bound for all cyclically $k$-edge-connected cubic graphs in the following theorem.
Note that for $k \le 5$, the bound from \hyperref[obs19]{Theorem~\ref{obs19}} is better. 
The next theorem has a faster algorithmic version, though.

\begin{theorem}\label{thm6}
Let $G$ be a cyclically $k$-edge-connected cubic graph $(k\geq 3)$. There exists an embedding of~$G$ with at most $\frac{n}{2k}$ singular edges.
Moreover, such an embedding can be found in time $O(n^\frac{3}{2}\log n)$. 
\end{theorem}

    \begin{proof}
        By the same technique as in the proof of \hyperref[obs19]{Theorem~\ref{obs19}}, it is enough to find two perfect matchings $M$ and $M'$ with $|M\cap M'| \leq \frac{n}{2k}$. Let $M$ be any perfect matching in $G$. (We can find this in time $O(n \log^2 n)$ as above, see \cite{Diks}.) Define a function $f:E(G) \to\mathbb{R}$ as follows
        \[f(e) = 
        \begin{cases} 
            \frac{1}{k} & \text{if}~~~   e\in M\\
    
            \frac{k-1}{2k} & \text{if}~~~ e \notin M.
        \end{cases}
        \] 

        We check that $f$ is in $\mathcal{P}_{PM}(G)$. Since the graph is cubic, around each vertex~$v$, one edge is in $M$ and has value $\frac{1}{k}$ and the other two edges connected to $v$ have value $\frac{k-1}{2k}$, and the sum of these three values is 1. It remains to show that for every set $U \subseteq V(G)$ with $|U|$ odd, 
        $x(\delta(U)) = \sum\limits_{e \in \delta(u)}f(e) \geq 1$. Then, by \hyperref[Edmonds]{Theorem~\ref{Edmonds}}, $f \in \mathcal{P}_{PM}(G)$. So, let $U$ be an odd-size vertex set in $G$. Then, we have one of the following cases:

        \begin{enumerate}
            \item $\delta(U) \subseteq M$. By removing the matching edges, every vertex in $U$ is incident with exactly two edges in $G[U]$. Therefore, it consists of a collection of cycles. As the graph is cyclically $k$-edge-connected, $|\delta(U)| \geq k$, So
            $$\sum\limits_{e\in \delta(U)} f(e) =
            \sum\limits_{e\in \delta(U)} \frac{1}{k} \geq
            k\cdot\dfrac{1}{k} = 1$$
            \item $|\delta(U)\backslash M| = 1$.
                We have the following two sub-cases. There exists a vertex $u$ with either degree one or two in $G[U]$ (see the picture below). We show that each sub-case is impossible to happen. Note that in the following pictures, bold edges are matching edges. 
                \begin{center}
                    \includegraphics[scale=.19]{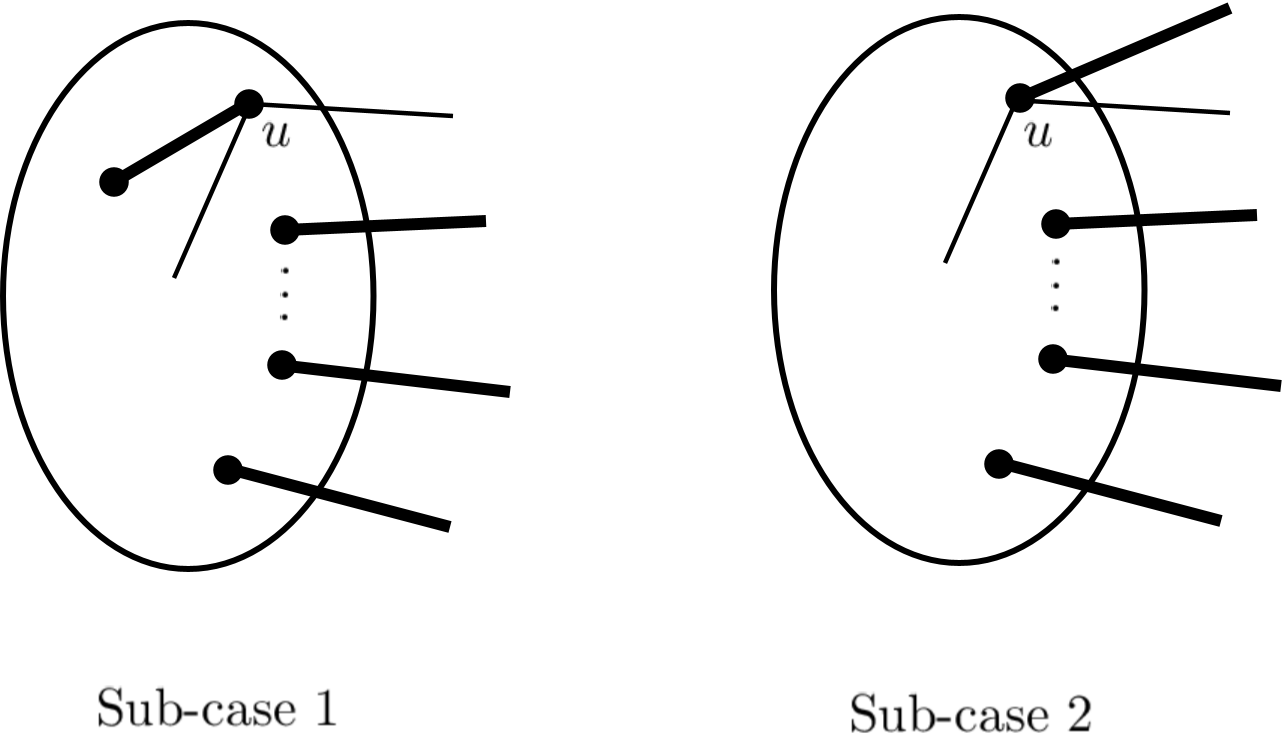}
                \end{center}
            In both sub-cases, there exists a vertex $u$ that has exactly one non-matching edge in $G[U]$. Let us remove the edges of $M$ from $G[U]$. Then, every vertex in $U \backslash \{u\}$ is incident with two edges in $G[U]$. So, $G[U]\backslash M$ has one vertex of degree $1$ and the rest of degree $2$, which is impossible.

            \item $|\delta(U)\backslash M| \geq 2$.
            In this case, as $|U|$ is odd and $|\delta(U)\backslash M| \geq 2$, we have $|\delta(U)| \geq 3$. If $\delta(U)$ has three non-matching edges, then 
            $$\sum\limits_{e\in \delta(U)} f(e) \geq 3\cdot\dfrac{k-1}{2k}\geq~1.$$
            Note that $k\geq 3$. If there is at least one matching edge, then 
            $$\sum\limits_{e\in \delta(U)} f(e) \geq 2\cdot\frac{k-1}{2k} + \frac{1}{k} = 1.$$
            
        \end{enumerate}

        As $f$ is in the perfect matching polytope, $f$ is a convex combination of 
        $\chi_{M_{i}}$, for some perfect matchings $M_i$. Set $S = E(G) \backslash M$. By definition of $f$, we have $f(S) = \frac{k-1}{2k}|S|$, hence 
        $$
          \chi_{M_i}(S) \geq \frac{k-1}{2k}|S| = \frac{2}{3}|E(G)| \frac{k-1}{2k} 
        $$
        for some perfect matching $M_i$ involved in the convex combination for $f$. Now,  
        $$|M \cup M_i| = |M| + |M_i \backslash M| \geq |E(G)|
        \Bigl(\frac{1}{3} + \frac{2}{3}\cdot\frac{k-1}{2k}\Bigr) = \frac{2k-1}{3k}|E(G)|$$ 
        Therefore, $|M\cap M_i| \leq \frac{n}{2k}$, and we can put $M' = M_i$. 
        This proves the existence of~$M'$. To find it efficiently in time 
        $O(n^\frac{3}{2}\log n)$, we use the algorithm of \cite{Duan2018}
        to find minimum weight perfect matching, as in the proof of \hyperref[sec3]{Corollary~\ref{cl18}}. 
    \end{proof}

\section{Conclusion}
We studied a new approach to the cycle double cover conjecture. We first introduced the notion of partial 
CDC, and then we showed that by having a partial CDC, we can 
extend it to an embedding in which some of the edges might have the same face on both sides. 
We called these edges singular and then provided nontrivial upper bounds on the minimum number of 
singular edges in an embedding of a cubic graph. We showed that these embeddings can 
be found in polynomial time in \hyperref[obs17]{Theorem~\ref{obs17}}, \hyperref[thm6]
{Theorem~\ref{thm6}}, \hyperref[obs19]{Theorem~\ref{obs19}} and  
\hyperref[cyclically 2k-edge-connected]{Theorem~\ref{cyclically 2k-edge-connected}}.
 It would be interesting to know what the limit of this technique is -- how many 
singular edges there must be in an extension of a given partial CDC.

In the future, it would be interesting to find better upper bounds on the number of singular edges. 
In particular, a sublinear estimate of the number of singular edges would be very interesting. 
Our hope is that this will guide us to a discovery of new
techniques that will eventually solve CDC conjecture, that is, to show that
there is an embedding with no singular edge. 

\section{Acknowledgment}
Special thanks should be given to Matt DeVos for helpful discussions that led to the extension of our results in \hyperref[sec3]{Section~\ref{sec3}}.

A preliminary version of this work was published in the proceedings of the International Workshop on Combinatorial Algorithms 2024, see \cite{Ghanbari}.

%
%
\bibliographystyle{alpha}
\bibliography{RN.bib}

\end{document}